\documentclass[12pt]{article}
\usepackage{amsmath}
\usepackage{amsmath,amsthm,amssymb,amscd}
\usepackage{hyperref}
\usepackage{color}
\usepackage{enumerate}

\numberwithin{equation}{section}

\newtheorem{thm}{Theorem}[section]

\newtheorem{lem}[thm]{Lemma}

\theoremstyle{definition}
\newtheorem{defn}[thm]{Definition}
\theoremstyle{proposition}
\newtheorem{prop}[thm]{Proposition}
\theoremstyle{remark}
\newtheorem{rem}[thm]{Remark}

\begin{document}

 \title{Existence of Nontrivial Solutions for the Nonlinear Equation on Locally Finite Graphs
}
\author{Ziliang Yang$^{\rm 1}$,\ \ Jiabao Su$^{\rm 2}$\thanks{Corresponding author.  $ \quad$ E-mail addresses: 2210501029@cnu.edu.cn (Z. Yang), sujb@cnu.edu.cn (J. Su), suncut@163.com (M. Sun).},\ \ Mingzheng Sun$^{\rm 3}$\\
{\small $^{\rm 1,2}$School of Mathematical Sciences, Capital Normal University,Beijing 100048, P. R. China} \\
 {\small $^{\rm 3}$College  of Science, North China University of Technology, Beijing 100144, P. R. China} }

 \maketitle
 \begin{abstract}
Suppose that $G=(V, E)$ be a locally finite and connected graph with symmetric weight and uniformly positive measure, where $V$ denotes the vertex set and $E$ denotes the edge set. We are concered with the following problem 
$$
\begin{cases}-\Delta u+h u=f(x, u), & \text { in } \Omega, \\ u=0, & \text { on } \partial \Omega,\end{cases}
$$
on the graph, where $h: \Omega \rightarrow \mathbb{R}$, $f: \Omega \times \mathbb{R} \rightarrow \mathbb{R}$ and $u: \Omega \rightarrow \mathbb{R}$. When $ f $ and $ h $ satisfies certain assumption conditions, we can ascertain the existence of one or two nontrivial solutions on the graph.
 \end{abstract}


{\bf MR(2010) Subject Classification:}  {35A15; 34B45; 58E05.}


\section{Introduction }

Let $G=(V, E)$ be a weighted graph where $V$ denotes the vertex set and $E$ denotes the edge set. We write $x \sim y$ (vertex $x$ is connected to vertex $y$) if $(x, y) \in E$. A graph $G$ is called connected if, for any vertices $x, y \in V$, there exists a path $\left\{x_k\right\}_{k=0}^n$ that satisfies $x=x_0 \sim x_1 \sim x_2 \sim \cdots \sim x_n=y$. 

Let $\mu_{x y}$ be the edge weights with $\mu_{x y}=\mu_{y x}>0$, we call it a symmetric weight on $G$. Any weight $ \mu_{x y} $ gives rise to a function on vertices as follows:
\begin{equation}
  \mu(x)=\sum_{y \sim x} \mu_{x y}.
\end{equation}
If each vertex has a finite number of edges, we say that $G$ is a locally finite graph. If the number of vertices is finite, we say that $G$ is a finite graph. A finite graph is locally finite.

 The measure $\mu: V \rightarrow \mathbb{R}^{+}$ on the graph is a finite positive function on $G$. We call it a uniformly positive measure if there exists a constant $\mu_{\min }>0$ such that $\mu(x) \geqslant \mu_{\min }$ for all $x \in V$. Consider a domain $\Omega \subset V$. The distance $d(x, y)$ of two vertices $x, y \in \Omega$ is defined by the minimal number of edges which connect these two vertices. If the distance $d(x, y)$ is uniformly bounded from above for any $x, y \in \Omega$, we call $\Omega$ a bounded domain in $V$. The boundary of $\Omega$ is defined as
$$
\partial \Omega:=\{y \notin \Omega: \exists x \in \Omega \text { such that } x y \in E\}
$$
and the interior of $\Omega$ is denoted by $\Omega^{\circ}$. Obviously, we have that $\Omega^{\circ}=\Omega$ which is different from the Euclidean case.

From \cite{1}, for any function $u: V \rightarrow \mathbb{R}$, define the $\mu$-Laplacian of $u$  by
\begin{equation}
  \Delta_{\mu} u(x)=\frac{1}{\mu(x)} \sum_{y \sim x} u(y) \mu_{xy}-u(x).
\end{equation}
Using $(1.1)$, we get the following Laplacian for short of $ u $:
\begin{equation}
	\Delta u(x)=\frac{1}{\mu(x)} \sum_{y \sim x} \mu_{x y}(u(y)-u(x)),
\end{equation}
The associated gradient form is
\begin{equation}
	\begin{aligned}
	    \Gamma(u, v)(x)
		& = \frac{1}{2}\{\Delta(u(x) v(x))-u(x) \Delta v(x)-v(x) \Delta u(x)\}  \\
		& =\frac{1}{2 \mu(x)} \sum_{y \sim x} \mu_{x y}(u(y)-u(x))(v(y)-v(x)).
	\end{aligned}
\end{equation}
Please refer to the appendix for details of $(1.2)-(1.4)$. Write $\Gamma(u)=\Gamma(u, u)$.
The length of the gradient for $u$ is
$$
|\nabla u|(x)=\sqrt{ \Gamma(u, u)(x)}=\left(\frac{1}{2\mu(x)} \sum_{y \sim x} \mu_{x y}(u(y)-u(x))^2\right)^{1 / 2}.
$$
For any function $f: \Omega \rightarrow \mathbb{R}$, an integral of $f$ over $\Omega$ is defined by
$$
\int_{\Omega} f d \mu=\sum_{x \in \Omega} \mu(x) f(x).
$$
In \cite{17}, the Lebesgue space $L^p(\Omega)$ on the graph $G$ is
$$
L^p(\Omega)=\left\{u: \Omega \rightarrow \mathbb{R}, \quad\|u\|_{L^p(\Omega)}<+\infty\right\}, \quad 1 \leq p \leq \infty,
$$
where the norm of $u \in L^p(\Omega)$ is given as
$$
\|u\|_{L^p(\Omega)}= 
\begin{cases}
	\left(\int_{\Omega}|u|^p d \mu \right)^{\frac{1}{p}}=\left(\sum_{x \in \Omega} \mu(x)|u(x)|^p\right)^{\frac{1}{p}}, & 1 \leq p<\infty, \\ \sup _{x \in \Omega}|u(x)|, & p=\infty.
\end{cases}
$$

Let $C_c(\Omega)$ be the set of all functions with compact support, and $W_0^{1,2}(\Omega)$ be the completion of $C_c(\Omega)$ under the norm
\begin{equation}
   \|u\|_{W_0^{1,2}(\Omega)}^2=\int_{\Omega \cup \partial \Omega}|\nabla u|^2 d \mu+\int_{\Omega} u^2 d \mu,
\end{equation}
that is,
$$
W_0^{1,2}(\Omega)=\overline{C_c(\Omega)}^
{\|u\|_{W_0^{1,2}(\Omega)}^2}.
$$
It is clearly that $W_0^{1,2}(\Omega)$ is a Hilbert space with the inner product
$$
\langle u, v\rangle=\int_{\Omega \cup \partial \Omega}\Gamma(u, v) d \mu+\int_{\Omega} uv d \mu, \quad \forall u, v \in W_0^{1,2}(\Omega) .
$$

In accordance with the groundwork laid by Grigor'yan et al.  \cite{10,11}, the Sobolev space $W_0^{1,2}(\Omega)$ and its norm on graphs are defined by
$$
W_0^{1,2}(\Omega)=\left\{u: \Omega \rightarrow \mathbb{R}|u|_{\partial \Omega}=0, \int_{\Omega \cup \partial \Omega}|\nabla u|^2 d \mu<+\infty\right\}
$$
and
\begin{equation}
\|u\|_{W_0^{1,2}(\Omega)}=\left(\int_{\Omega \cup \partial \Omega}|\nabla u|^2 d \mu\right)^{\frac{1}{2}}.
\end{equation}

Motivated by papers \cite{2,4,24}, for all $x \in \Omega$, we study the following  equation on the graph

\begin{equation}
	\begin{cases}-\Delta u+h u=f(x, u), & \text { in } \Omega, \\ u=0, & \text { on } \partial \Omega,\end{cases}
\end{equation}
where $h: \Omega \rightarrow \mathbb{R}$, $f: \Omega \times \mathbb{R} \rightarrow \mathbb{R}$ and $u: \Omega \rightarrow \mathbb{R}$. $ \lambda_1(\Omega) $ is the first eigenvalue of the Laplacian with respect to Dirichlet boundary condition, which is defined by
\begin{equation}
   \lambda_1(\Omega)=\inf _{u \neq 0,\left.u\right|_{\partial \Omega}=0} \frac{\int_{\Omega \cup \partial \Omega}|\nabla u|^2 d \mu}{\int_{\Omega} u^2 d \mu}
\end{equation}

The study of discrete weighted Laplacians and an assortment of equations on graphs has garnered considerable scholarly interest of late. In \cite{2}, Grigor'yan, Lin and Yang studied nonlinear Schrödinger equations
\begin{equation}
   -\Delta u+h(x) u=f(x, u) \quad \text { in } V
\end{equation}
on a connected locally finite graph $G$.
In \cite{18}, Zhang and Zhao established the existence and convergence of ground state solutions for (1.9), when $b(x)=\lambda a(x)+1$ and $f(x, u)=|u|^{p-1} u$. In \cite{11}, Grigoryan, Lin and Yang proved that there exists a positive solution to
$$
\left\{\begin{aligned}
	-\Delta u-\alpha u & =|u|^{p-2} u, & & \text { in } \Omega^{\circ}, \\
	u & =0, & & \text { on } \partial \Omega.
\end{aligned}\right.
$$ 
In \cite{6}, Shou consider the nonlinear Schrödinger equations on the graph $ G $, where $ G $ satisfies the curvature-dimension type inequality.

On an Euclidean space, the standing wave solution to the Schrödinger equation was  researched by Rabinowitz \cite{14}. Furthermore, the existence of solutions under various conditions has been critically examined by acclaimed researchers such as Bartsch-Willem \cite{15},  Wang \cite{3}, Rabinowitz-Su-Wang \cite{4}, Li \cite{16} and others. 

On the other hand, research has demonstrated that established approaches for examining the existence and multiplicity of solutions for partial differential equations delineated on open domains in Euclidean spaces can be equivalently applicable to partial differential equations on fractals. Using the mountain pass theorem and the saddle point theorem, Falconer-Hu \cite{13} obtained existence results on the Sierpiński gasket. Specific attention has been given to partial differential equations  on the Sierpinski gasket. Reference is made to the groundbreaking work as stated in paper \cite{20}, which discussed nonlinear elliptic equations on the Sierpinski gasket in a two-dimensional Euclidean space. For the Sierpiński gasket, we can also refer to \cite{21,22,23,24,25} and so on.

Now, we define a space of functions
$$
H=\left\{u \in W_0^{1,2}(\Omega): \int_{\Omega} h u^2 d \mu<+\infty\right\}
$$
with a norm
\begin{equation}
  \|u\|_{H}=\left(\int_{\Omega \cup \partial \Omega}|\nabla u|^2 d \mu+\int_{\Omega} hu^2 d \mu\right)^{1 / 2} 
\end{equation}
on the graph. We see that $H$ is also a Hilbert space with the inner product
$$
\langle u, v\rangle_{H}=\int_{\Omega \cup \partial \Omega}\Gamma(u, v) d \mu+\int_{\Omega} huv d \mu, \quad \forall u, v \in H .
$$

We define a functional on $H$ by
\begin{equation}
  \Phi\left(u\right)=\frac{1}{2} \left(\int_{\Omega \cup \partial \Omega} |\nabla u|^2 d \mu+ \int_{\Omega}h u^2 d \mu\right)-\int_{\Omega} F(x, u) d \mu,
\end{equation}
where $   F(x, u)=\int_0^u f(x, t) d t$ is the primitive function of $f$. Thus, we have

\begin{equation}
   \left\langle\Phi^{\prime}\left(u\right), u\right\rangle =
   \int_{\Omega \cup \partial \Omega} |\nabla u|^2 d \mu+ \int_{\Omega}h u^2 d \mu-\int_{\Omega} f\left(x, u\right) u d \mu.
\end{equation}

The main results of this paper are the following:

\begin{thm} 
	 Suppose that $G=(V, E)$ be a locally finite and connected graph with symmetric weight and uniformly positive measure. Let $h: \Omega \rightarrow \mathbb{R}$ be a function satisfying the hypotheses
	
	$\left(H_1\right)$ there exists a constant $h_0>0$ such that $h(x) \geq h_0$ for all $x \in \Omega$;
	
    $\left(H_2\right)$ $ 1 / h \in L^1(\Omega)$.
    
	Suppose that $f: \Omega \times \mathbb{R} \rightarrow \mathbb{R}$ satisfies the following hypotheses:
	
	$\left(F_1\right)$  $ f \in C^1(\Omega \times \mathbb{R}, \mathbb{R})$
	
	$\left(F_2\right)$  $f(x, 0)=0=f_u(x, 0)$.
	
    $\left(F_3\right)$ There is $C>0$ such that
	 \begin{equation}
	   |f(x, u)| \leq C\left(1+|u|^{p-1}\right), \text { for some } 2< p <2^*= \begin{cases}+\infty, & N=1,2, \\ \frac{2 N}{N-2}, & N=3,\end{cases}
	 \end{equation}
	 where $ x \in \Omega $ and $ u \in \mathbb{R} $.

	$\left(F_4\right)$ There is $\theta>2, M>0$ such that
	\begin{equation}
	  0<\theta F(x, u) \leq u f(x, u) \quad x \in \Omega,\ |u| \geq M,
	\end{equation}
	where $F(x, u):=\int_0^u f(x, t) d t$.
	
   Hence, equation $(1.7)$ is guaranteed to have at least one nontrivial solution.
\end{thm}

\begin{rem}
	The assumption conditions of Theorem 1.1 were established by drawing on the assumption conditions put forth in references \cite{2,4}. The condition $\left(F_4\right)$ pertains to the widely recognized Ambrosetti-Rabinowitz type condition as defined in reference \cite{12}. Observe that the Ambrosetti-Rabinowitz condition $\left(F_4\right)$ implies 
    \begin{equation}
	 \lim_{u\rightarrow
		\infty}\frac{F\left(x,u\right)}{u^2}=+\infty. 
	 \end{equation}
	Therefore $F$ grows at a “ super-quadratic” rate.
\end{rem}

\begin{thm}
	Suppose that $G=(V, E)$ be a locally finite and connected graph with symmetric weight and uniformly positive measure. Let $h: \Omega \rightarrow \mathbb{R}$ be a function satisfying 	$\left(H_1\right)$ and $ \left(H_2\right) $.

	Suppose that $f: \Omega \times \mathbb{R} \rightarrow \mathbb{R}$ satisfies $\left(F_1\right)$,$ \left(F_3\right) $ and
	
	$\left(F_5\right)  \frac{f\left(x,u\right)}{u}$ is nondecreasing in $ u>0 $, for all $ x \in \Omega$.
	
	$\left(F_6\right) \lim_{u\rightarrow \infty}\frac{f\left(x,u\right)}{u}=+\infty $.
	
	Hence, equation $(1.7)$ is guaranteed to have at least one nontrivial solution.
\end{thm}

\begin{rem}
	The condition $\left(F_6\right)$ is weaker than the condition $\left(F_4\right)$, and (1.15) can also be derived from the condition $\left(F_6\right)$. 
\end{rem}

\begin{thm}
   Suppose that $G=(V, E)$ be a locally finite and connected graph with symmetric weight and uniformly positive measure. Let $h: \Omega \rightarrow \mathbb{R}$ be a function satisfying 	$\left(H_1\right)$ or
  
   $\left(H_3\right)$   
   $$
  	\int_{\Omega} h(x) d \mu \leq\frac{1}{\mu_{\min} h_0}
   $$
  and $ \left(H_2\right) $.
   Suppose that $f: \Omega \times \mathbb{R} \rightarrow \mathbb{R}$ satisfies $\left(F_1\right)$,$ \left(F_3\right) $,$ \left(F_4\right) $ and
   
   $\left(F_7\right)$ $f(x, 0) \neq 0$ for any $ x \in \Omega. $
   
  Hence,  for any $\rho>0$ and there exist a positive constant $\beta$ such that
  \begin{equation}
     1<\beta+1 \leq \frac{\rho}{2 \max_{\substack{x \in \Omega \\|u| \leq \kappa \sqrt{\rho}}}\left|\int_0^u f(x, t) d t\right|},
  \end{equation}
   where 
   \begin{equation}
      	\kappa:= \begin{cases}
      		\left(\mu_{\min} h_0\right)^{\frac{1}{2}}, & \text { if $( H_1)$ holds, } \\
      	 \frac{\left(\mu_{\min} h_0\right)^{\frac{1}{2}}}{\sqrt{1-\mu_{\min} h_0 \int_{\Omega}h(x) d \mu}}, & \text { if $(H_3)$ holds, }
      	\end{cases}
   \end{equation}
   equation $(1.7)$ is guaranteed to have at least two nontrivial solutions one of which lies in 
   $$
   \mathbb{B}_{\rho}:=\left\{u \in W_0^{1,2}(\Omega):\|u\|_H<\sqrt{\rho}\right\} .
   $$
\end{thm}
	
	\begin{rem}
		Theorem 1.5 is an extension of Theorem 1.1. In Theorem 1.5 and the following Theorem 1.7, our requirements for the condition $ h(x) $ are more stringent than the assumption conditions in reference \cite{24}. Moreover, $\left(H_3\right)$, (1.16), and (1.17) are all new assumption conditions. The condition $\left(F_7\right)$ assures that all the solutions of problem (1.7), if any, are nontrivial. While these two theorems bear structural similarities, they differ in terms of the specific demands on assumptions and the distinctiveness of their results. The subtle distinctions they exhibit are indeed notable.
	\end{rem}

\begin{thm}
	Suppose that $G=(V, E)$ be a locally finite and connected graph with symmetric weight and uniformly positive measure. Let $h: \Omega \rightarrow \mathbb{R}$ be a function satisfying 	$\left(H_1\right)$ and $ \left(H_2\right) $.
	Suppose that $f: \Omega \times \mathbb{R} \rightarrow \mathbb{R}$ satisfies $\left(F_1\right)$, $ \left(F_3\right) $, $ \left(F_4\right) $, $ \left(F_7\right) $and
	
	$\left(F_8\right)$ 
	$$
     \max_{(x, u) \in \Omega \times\left[-M_0, M_0\right]}\left|\int_0^u f(x, t) d t\right| \leq \frac{M_0^2}{2(\beta+1) \mu_{\min} h_0}.
	$$
	
	Hence, equation $(1.7)$ is guaranteed to have at least two nontrivial solutions one of which lies in 
	$$
	   \mathbb{B}_{\frac{M_0^2}{\mu_{\min} h_0}}:=\left\{u \in W_0^{1,2}(\Omega):\|u\|_H<\sqrt{\frac{M_0^2}{\mu_{\min } h_0}}\right\} .
	$$
\end{thm}

\begin{rem}
	Theorem 1.7 is a special case of Theorem 1.5. In Theorem 1.7, both condition $\left(F_8\right)$ and $ \mathbb{B}_{M_0^2/\left(\mu_{\min} h_0\right)} $ are novel.
\end{rem}

The paper is organized as follows. In section 2, we give some preliminaries. In section 3, the proof Theorem 1.1 is given. In section 4, we proof Theorem 1.3. The proof of Theorem 1.5 is documented in Section 5, whereas Theorem 1.7 is substantiated in Section 6. In Section 7, We furnish appendices pertinent to the Introduction Section $(1.2)-(1.4)$. In this paper, the letter $C$, $ c_0 $ or $c$ will be used to denote various positive constant.

\section{Preliminaries}

Henceforth, let's assume that $G=(V, E)$ be a locally finite and connected graph with symmetric weight and uniformly positive measure.
In this section, we need the following some important facts.
\begin{lem}$(\cite{8,9})$ 
	If any sequence $\left\{u_n\right\} \subset H_0^1(\Omega)$ with
	$$
	\Phi\left(u_n\right) \text { being bounded, }\ \Phi^{\prime}\left(u_n\right) \rightarrow 0(n \rightarrow \infty),
	$$
	has a convergent subsequence, then $\Phi$ satisfies the $(Palais-Smale)$ condition. If
	$$
	\left|\Phi\left(u_n\right)\right| \leq C \text {, for all } n \text { and some constant } C, \ \Phi^{\prime}\left(u_n\right) \rightarrow 0(n \rightarrow \infty),
	$$
	then $\left\{u_n\right\} \subset H_0^1(\Omega)$ is a $(PS)$ sequence for $\Phi$.
	
\end{lem}

We say that $\Phi$ satisfies (Palais-Smale) (for short (PS)) condition if $\Phi$ satisfies $(\mathrm{PS})_c$ at any $c \in \mathbb{R}$. A weaker version of the (PS) condition can be derived from Lemma 2.1 as follows. 

Let $\Phi \in C^1(X, \mathbb{R})$. We say that $\Phi$ satisfies the Cerami condition at the level $c \in \mathbb{R}\left((C)_c\right.$ in short $)$ if any sequence $\left\{u_n\right\} \subset X$ such that
$$
\Phi\left(u_n\right) \rightarrow c, \quad\left(1+\left\|u_n\right\|\right)\left\|\Phi^{\prime}\left(u_n\right)\right\| \rightarrow 0 \text { as } n \rightarrow \infty
$$
has a convergent subsequence.

Additionally, we have the following Lemma.

\begin{lem}$(\cite{11,12,26})$ 
	 Let $(X,\|\cdot\|)$ be a Banach space, $\Phi \in C^1(X, \mathbb{R}),\ e \in X$ and $r>0$ be such that $\|e\|>r$ and
	$$
	b:=\inf_{\|u\|=r} \Phi(u)>\Phi(0) \geq \Phi(e) .
	$$
	If $\Phi$ satisfies the $(P S)_c$ condition with $c:=\inf_{\gamma \in \Gamma} \max_{t \in[0,1]} \Phi(\gamma(t))$, where
	$$
	\Gamma:=\{\gamma \in C([0,1], X): \gamma(0)=0, \gamma(1)=e\},
	$$
	then $c$ is a critical value of $\Phi$.
	\end{lem}

It is evident, from the content stated in (1.5) and (1.6) of the Introduction, that the following lemma can be clearly inferred. 
The process of proof  is novel.
\begin{lem}
	The norm $(1.5)$ is equivalent to $(1.6)$.
\end{lem}

\begin{proof}
	By (1.8), we have
	$$
	\lambda_1(\Omega) \int_{\Omega} u^2 d \mu \leq \int_{\Omega \cup \partial \Omega}|\nabla u|^2 d \mu .
	$$
	Hence
	$$
	\begin{aligned}
		\|u\|_{W_0^{1,2}(\Omega)}^2 & =\int_{\Omega \cup \partial \Omega}|\nabla u|^2 d \mu+\int_{\Omega} u^2 d \mu \\
		& \leq \int_{\Omega \cup \partial \Omega}|\nabla u|^2 d \mu+\frac{1}{\lambda_1(\Omega)} \int_{\Omega \cup \partial \Omega}|\nabla u|^2 d \mu \\
		& =\left(1+\frac{1}{\lambda_1(\Omega)}\right) \int_{\Omega \cup \partial \Omega}|\nabla u|^2 d \mu .
	\end{aligned}
	$$
	On the other hand,
	$$
	\begin{aligned}
		\|u\|_{W_0^{1,2}(\Omega)}^2& =\int_{\Omega \cup \partial \Omega}|\nabla u|^2 d \mu \\
		& =\sum_{x \in \Omega \cup \partial \Omega} \mu(x)|\nabla u|^2(x) \\
		& \leq \sum_{x \in \Omega \cup \partial \Omega} \mu(x)\left(|\nabla u|^2(x)+u^2(x)\right) \\
		& =\int_{\Omega \cup \partial \Omega}|\nabla u|^2 d \mu+\int_{\Omega \cup \partial \Omega} u^2 d \mu \\
		& =\int_{\Omega \cup \partial \Omega}|\nabla u|^2 d \mu+\int_{\Omega} u^2 d \mu .
	\end{aligned}
	$$
	In conclusion, the validation of this lemma has been established.
\end{proof}

We next delve into the integration by parts formulas on graphs, which are crucial when employing methods in the calculus of variations. The proofs for the subsequent one lemma can be located in \cite{18}, and hence, are not reiterated here.

\begin{lem}
	Suppose that $u \in W^{1,2}(V)$ and $\Omega \subset V$ is a bounded connected domain. Then for any $v \in C_c(\Omega)$, we have
	$$
	\int_{\Omega \cup \partial \Omega} \nabla u \nabla v d \mu=\int_{\Omega \cup \partial \Omega} \Gamma(u, v) d \mu=-\int_{\Omega}(\Delta u) v d \mu .
	$$
\end{lem}

Now we can define the weak solution  of the equation (1.7) as

\begin{defn}
    Suppose $ u \in H $. If for any $ \varphi \in H $, there holds
	$$
	\int_{\Omega \cup \partial \Omega}\Gamma(u, \varphi) d \mu
	+\int_{\Omega}h u \varphi d \mu=\int_{\Omega} f(x, u) \varphi d \mu, 
	$$
	then $u$ is called a weak solution of (1.7).
\end{defn}

\begin{prop}
	If $u \in H$ is a weak solution of $(1.7)$, then $u$ is also a point-wise solution of $(1.7)$.
\end{prop}
\begin{proof}
	If $u$ is a weak solution, then by Lemma 2.4, we have 
	\begin{equation}
	  \int_{\Omega}(-\Delta u+h u) \varphi d \mu=\int_{\Omega} f(x, u) \varphi d \mu, \quad \forall \varphi \in C_c(\Omega) .
	\end{equation}
	For any fixed $x_0 \in \Omega$, taking a test function $\varphi: \Omega \rightarrow \mathbb{R}$ in (2.1) with
	$$
	\varphi(x)=\frac{\eta\left(x-x_0\right)}{\sum_{x \in \Omega} \eta\left(x-x_0\right)}, 
	$$
	where
		$$
	 \eta(x)=\left\{\begin{array}{cc}
		e^{-1}, & x=0, \\
		0, & x \neq 0.
	\end{array}\right.
	$$
	We have
	$$
	-\Delta u(x_0)+h(x_0) u(x_0)-f(x_0, u(x_0))=0 .
	$$
	Since $x_0$ is arbitrary, we conclude that $u$ is a point-wise solution of $(1.7)$.
\end{proof}

It merits noting that there is no uniqueness in the choice of test function, with further details present in Reference \cite{2,19}. Proposition 2.6 choose a new test function.

Ultimately within this section, we arrive at the discussion regarding the Sobolev embedding theorems on the graph. The next lemma is similar to the results in \cite{2,18,19}.

\begin{lem}
	Suppose that $G=(V, E)$ be a locally finite and connected graph with symmetric weight and uniformly positive measure. Then $H$ is continuously embedded into $L^q(\Omega)$ for any $q \in[1,+\infty]$. Namely, there exists a constant $C$ depending only on $\mu(x), \mu_{\min}, q, \Omega$ and $ h_0 $ such that
	$$
	\|u\|_{L^q(\Omega)} \leq C\|u\|_H.
	$$
	Moreover, for any bounded sequence $\left\{u_n\right\} \subset H$, there exists $u \in H$ such that, up to a subsequence
	$$
	\left\{\begin{array}{l}
		u_n \rightarrow u, \quad \text { in } H, \\
		u_n(x) \rightarrow u(x), \quad \forall x \in \Omega, \\
		u_n \rightarrow u,\quad \text { in } L^q(\Omega).
	\end{array}\right.
	$$
\end{lem}

\begin{proof}
	For any $u \in H$ and vertex $x_0 \in \Omega$, we have
	$$
	\begin{aligned}
		\|u\|_H^2 & =\int_{\Omega \cup \partial \Omega}|\nabla u|^2 d \mu+\int_{\Omega} h u^2 d \mu \\
		& \geq \int_{\Omega} h u^2 d u=\sum_{x \in \Omega} h(x) u^2(x) \mu(x) \\
		& \geq \mu_{\min } h_0 u^2\left(x_0\right),
	\end{aligned}
	$$
	thereby obtaining
	$$
	u\left(x_0\right) \leq \left(\frac{1}{\mu_{\min} h_0}\right)^{\frac{1}{2}}\|u\|_H.
	$$
	Thus
	$$
	\|u\|_{L^{\infty}(\Omega)} \leq \left(\frac{1}{\mu_{\min} h_0}\right)^{\frac{1}{2}}\|u\|_H .
	$$
	Thus for any $1 \leq q <+\infty$, we have
	$$
	\begin{aligned}
		\|u\|_{L^q(\Omega)} & =\left(\sum_{x \in \Omega} \mu(x)|u(x)|^q\right)^{\frac{1}{q}} \\
		& \leq\left(\sum_{x \in \Omega} \mu(x)\right)^{\frac{1}{q}}\left(\frac{1}{\mu_{\min} h_0}\right)^{\frac{1}{2}}\|u\|_H:=C\|u\|_H.
	\end{aligned}
	$$
	
	Given that the Hilbert space $H$ is reflexive, it follows that for any bounded sequence $\left\{u_k\right\}$ in $H$, there exists a subsequence such that 
    $u_n \rightarrow u$ in $H$. In particular,
	$$
	\lim_{n \rightarrow \infty} \int_{\Omega} u_n \varphi d u=\int_{\Omega} u \varphi d u, \quad \forall \varphi \in C_c(\Omega),
	$$
	that is, 
	$$
	\lim_{n \rightarrow \infty} \int_{\Omega}\left(u_n-u\right) \varphi d \mu=\lim_{n \rightarrow \infty} \sum_{x \in \Omega} \mu(x)\left(u_n(x)-u(x)\right) \varphi(x)=0 .
	$$
	Let
	$$
	\varphi(x)=\left\{\begin{array}{cc}
		u_n\left(x_0\right)-u\left(x_0\right), & x=x_0, \\
		0, & x \neq x_0,
	\end{array}\right.
	$$
	thus
	$$
	\lim_{n \rightarrow \infty} \mu\left(x_0\right)\left(u_n\left(x_0\right)-u\left(x_0\right)\right)=0,
	$$
	which implies that $\lim_{n \rightarrow \infty} u_n(x)=u(x)$ for any $x \in \Omega$.
	
	Now, we commence the proof showing that $u_n \rightarrow u$ in $L^q(\Omega)$ for all $1 \leq q \leq \infty$. 
	
	  Since $\left\{u_n\right\}$ is bounded in $H$ and $u \in H$, we can deduce that
	$$
	\int_{\Omega}\left|u_n-u\right| d \mu \leq C_1.
	$$
	on the one hand, for any $\varepsilon>0$,
	\begin{equation}
	   \int_{\operatorname{dist}\left(x, x_0\right)>R}\left|u_n-u\right| d u \leq \varepsilon \int_{\Omega}\left|u_n-u\right| d u \leq  C_1\varepsilon.
	\end{equation}
	On the other hand, since $\left\{x \in \Omega: \operatorname{dist}\left(x, x_0\right)>R\right\}$ is a finite set and $u_n(x) \rightarrow u(x)$ for any $x \in \Omega$ as $n \rightarrow \infty$, we have
	\begin{equation}
		\lim_{n \rightarrow+\infty} \int_{\operatorname{dist}\left(x, x_0\right) \leq R}\left|u_n-u\right| d \mu=0.
	\end{equation}
	Combining (2.2) and (2.3), we derive that
	$$
	\lim_{n \rightarrow \infty} \int_{\Omega}\left|u_n-u\right| d \mu=0 .
	$$
	In particular, it holds that up to a subsequence, $u_n \rightarrow u$ in $L^{q}(\Omega)$. Since $\mu(x) \geq \mu_{\min }>0$, we have
	$$
	\begin{aligned}
		h \int_{\Omega}\left|u_n-u\right| d u & =\sum_{x \in \Omega} \mu(x)h(x)\left|u_n(x)-u(x)\right| \\
		& \geq \mu_{\min} h_0\left|u_n(x)-u(x)\right|,
	\end{aligned}
	$$
	thus
	$$
	\left|u_n(x)-u(x)\right| \leq \frac{h}{\mu_{\min} h_0} \int_{\Omega}\left|u_n-u\right| d \mu,
	$$
	Hence
	$$
	\left\|u_n-u\right\|_{L^{\infty}(\Omega)} \leq \frac{h}{\mu_{\min} h_0} \int_{\Omega}\left|u_n-u\right| d \mu .
	$$
	We have for any $1<q<+\infty$,
	$$
	\begin{aligned}
		\int_{\Omega}\left|u_n-u\right|^q d \mu =& \int_{\Omega}\left|u_n-u\right|^{q-1} \cdot\left|u_n-u\right| d \mu \\
		\leq & \left(\frac{h}{\mu_{\min} h_0}\right)^{q-1}\left(\int_{\Omega}\left|u_n-u\right| d \mu\right)^{q-1} \int_{\Omega}\left|u_n-u\right| d \mu \\
		= & \left(\frac{h}{\mu_{\min} h_0}\right)^{q-1}\left(\int_{\Omega}\left|u_n-u\right| d \mu\right)^q .
	\end{aligned}
	$$
	Therefore, up to a subsequence, $u_n \rightarrow u$ in $L^q(\Omega)$ for all $1 \leq q \leq+\infty$. The proof is completed.

\end{proof}

Finally we have the following Lemma 2.8 (refer to \cite{25}). Lemma 2.8 is derived through the conjoint implementation of the classical Pucci-Serrin Theorem (refer to \cite{26}) and a result pertaining to local minimum credited to Ricceri (refer to \cite{27}). The particulars are delineated as follows.

\begin{lem}
	Let $X$ be a reflexive real Banach space and let $\Theta, \Psi: X \rightarrow \mathbb{R}$ be two continuously Gâteaux differentiable functionals such that
	
	- $\Theta$ is sequentially weakly lower semicontinuous and coercive in $X$.
	
	- $\Psi$ is sequentially weakly continuous in $X$.
	
	In addition, assume that for each $\zeta>0$ the functional $\Phi_\zeta:=\zeta \Theta-\Psi$ satisfies the $(PS)$ condition. Then, for each $\rho>\inf_X \Theta$ and each
	$$
	\zeta>\inf_{u \in \Theta^{-1}((-\infty, \rho))} \frac{\left(\sup _{v \in \Phi^{-1}((-\infty, \rho))} \Psi(v)\right)-\Psi(u)}{\rho-\Theta(u)}
	$$
	the following alternative holds: either the functional $\Phi_\zeta$ has a strict global minimum which lies in $\Theta^{-1}((-\infty, \rho))$, or $\Phi_\zeta$ has at least two critical points one of which lies in $\Theta^{-1}((-\infty, \rho))$.
\end{lem}

\section{Proof of Theorem 1.1}

In this section, we establish Theorem 1.1, leveraging Lemma 2.1 and Lemma 2.2. We initially present the following.

  \begin{lem}
  	There exist $ \delta, r >0 $ such that $ \Phi\left(u\right)\geq \delta $ when $ \|u\|_H=r $.
  	
  \end{lem}

  \begin{proof}
  	By $\left(F_2\right)$ and $\left(F_3\right)$, for $ 0<\varepsilon <C $, there is $ C_\varepsilon >0 $ such that
  	\[
  	   F\left(x,u\right)\leq\frac{\varepsilon}{2}u^2+C_\varepsilon |u|^p.
  	\]
  	
  	Take $ \|u\|_H\leq \left(\frac{C-\varepsilon}{4C C_\varepsilon }\right)^{\frac{1}{p-2}} $. Since $ p>2 $ and Lemma 2.7, we have
  	
  	 \begin{equation}
  		\begin{aligned}
  			\Phi\left(u\right)
  			& = \frac{1}{2}\left(\int_{\Omega \cup \partial \Omega}|\nabla u|^2 d \mu+\int_{\Omega} h u^2 d \mu\right)-\int_{\Omega} F\left(x,u\right) d \mu \\
  			& \geq \frac{1}{2} \|u\|_H^2-\int_{\Omega} \left(\frac{\varepsilon}{2}u^2+C_\varepsilon |u|^p\right) d \mu \\
  			& \geq \left(\frac{1}{2}-\frac{\varepsilon}{2C}\right)\|u\|_H^2-C_\varepsilon \|u\|_H^p\\
  			& \geq \frac{C-\varepsilon}{4 C}r^2:=\delta,
  		\end{aligned}
  	\end{equation}
  then we obtain the conclusion. The proof is completed.
  
  \end{proof}

Subsequently, we introduce the following.

  \begin{lem}
  	There exists some nonnegative function $ u \in H $ such that $\Phi (tu) \rightarrow-\infty$ as $ t\rightarrow +\infty $.
  \end{lem}

\begin{proof}
	While this proof approach bears resemblance to that in reference \cite{2}, it also exhibits distinct differences.
	
	By $\left(F_4\right)$, we have
   \begin{equation}
      F(x, u) \geq e^{-c}|u|^{\theta}-C, \quad \forall |u|\geq M, \quad (x,u) \in \Omega \times [0,+\infty).
   \end{equation} 
   Indeed, equation (1.14) can be reformulated as follows
    $$
    u \frac{\partial F(x, u)}{\partial u} \geq \theta F(x, u), \quad \forall |u| \geq M .
    $$
    If $u \geq M$, then
    $$
    \theta \frac{\partial u}{u} \leq \frac{\partial F(x, u)}{\partial u}.
    $$
    Integrating the above inequality from $M$ to $u$ yields
    $$
     \theta (\ln u-\ln M) \leq \ln F(x, u)-\ln F(x, M),
    $$
    that is,
    $$
    	\ln F(x, u) \geq \theta \ln u-C, 
    $$
    where $  C: =\theta \ln M-\ln F(x, M)  $.
    
    For any $u \leq-M$, then
    $$
      \theta \frac{\partial u}{u} \geq \frac{\partial F(x, u)}{F(x, u)}.
    $$
    Integrating the above inequality from $u$ to $-M$ yields
    $$
       \theta (\ln M-\ln |u|) \geq \ln F(x,-M)-\ln F(x, u),
    $$
    that is,
    $$
    \ln F(x, u) \geq \theta \ln |u|-C, \quad \forall u \leq -M,
    $$
    where $  C: = \theta\ln M-\ln F(x, -M)  $. 
    
    Overall, we deduce that
    $$
    \ln F(x, u) \geq \theta \ln |u|-C, \quad \forall |u| \geq M.
    $$
    Hence, (3.2) is valid.
    
    Let $x_0 \in \Omega$ be fixed. Take a function 
    $$
    u(x)= \begin{cases}1, & x=x_0, \\ 0, & x \neq x_0.\end{cases}
    $$
    As $t \rightarrow+\infty$, given that $\theta>2$ and $ G $ is locally finite, then we have
    $$
    \begin{aligned}
    	\Phi(t u) & =\frac{t^2}{2} \int_{\Omega \cup \partial \Omega}|\nabla u|^2 d \mu+\frac{t^2}{2} \int_{\Omega} h u^2 d \mu
    	-\int_{\Omega} F\left(x,tu\right) d \mu \\
    	& =\frac{t^2}{2} \sum_{x \in{\Omega \cup \partial \Omega}} \mu(x)|\nabla u|^2(x)+\frac{t^2}{2}\sum_{x \in{\Omega}} \mu\left(x\right) h\left(x\right) u^2(x)-\sum_{x \in{\Omega}}\mu\left(x\right) F\left(x, tu(x)\right) \\
    	& \leq \frac{t^2}{2} \sum_{x \in {\Omega \cup \partial \Omega}} \mu(x)|\nabla u|^2(x)+\frac{t^2}{2} \mu\left(x_0\right) h\left(x_0\right)-e^{-c}\mu(x_0) |t|^{\theta}+C \mu\left(x_0\right) \\
    	& \rightarrow-\infty.
    \end{aligned}
    $$
	The proof is completed.
	\end{proof}

 \begin{lem}
    If $h$ satisfies $\left(H_1\right)$ and $\left(H_2\right), f$ satisfies $\left(F_1\right)$ and $\left(F_4\right)$, then $\Phi$ satisfies the $(PS)_c$ $(Palais-Smale)$ condition for any $c \in \mathbb{R}$.
\end{lem}

\begin{proof}
	We now prove that $ \Phi \left(u_n\right) $ satisfies the $ (\mathrm{PS})_c $ condition for any $ c \in \mathbb{R} $.
	Let $ \{u_n\} \in W_0^{1,2}\left(\Omega\right)$ be such that
	\begin{equation}
	   \Phi\left(u_n\right)\rightarrow c, \  \Phi^{\prime}\left(u_n\right) \rightarrow 0 \ (n \rightarrow \infty).
    \end{equation}
    This implies of course that there is a constant $ C>0 $ such that
    \begin{equation}
      \left|\frac{1}{2} \left(\int_{\Omega \cup \partial \Omega}|\nabla u_n|^2 d \mu+\int_{\Omega} h u_n^2 d \mu\right)-\int_{\Omega} F\left(x,u_n\right) d \mu \right| \leq C,
    \end{equation}
    \begin{equation}
        \left|\int_{\Omega \cup \partial \Omega}\Gamma\left(u_n, \varphi\right) d \mu
        +\int_{\Omega}h u_n \varphi d \mu-\int_{\Omega} f\left(x, u_n\right) \varphi d \mu\right| \leq \epsilon_k\|\varphi\|_H, \quad \forall
        \varphi \in H,
    \end{equation}
    where $ \varepsilon_k \rightarrow 0 $ as $ k\rightarrow +\infty$.
    
    According to Lemma 2.1, we merely need to demonstrate that $ \{u_n\} $ is bounded. By (3.4) and (3.5), for large $ n \in \mathbb{N} $, we have
    
   \begin{equation}
   	\begin{aligned}
   		& C+o(1)\left\|u_n\right\|_H \\
   		& \geq \Phi\left(u_n\right)-\frac{1}{\theta}\left\langle\Phi^{\prime}\left(u_n\right), u_n\right\rangle \\
   		&= \frac{1}{2} \int_{\Omega \cup \partial \Omega}|\nabla u_n|^2 d \mu+ \frac{1}{2}\int_{\Omega} h u_n^2 d \mu-\int_{\Omega} F\left(x,u_n\right) d \mu \\
   		& \quad \quad-\frac{1}{\theta}\left(\int_{\Omega \cup \partial \Omega}\Gamma\left(u_n, u_n\right) d \mu
   		+\int_{\Omega}h u_n^2 d \mu-\int_{\Omega} f\left(x, u_n\right) u_n d \mu\right) \\
   		&=\left(\frac{1}{2}-\frac{1}{\theta}\right) \left(\int_{\Omega \cup \partial \Omega}|\nabla u_n|^2 d \mu+\int_{\Omega} h u_n^2 d \mu\right) \\
   		& \quad \quad -\frac{1}{\theta} \int_{\Omega}\left(\theta F\left(x, u_n\right)-f\left(x, u_n\right) u_n\right) d \mu \\
   		& \geq\left(\frac{1}{2}-\frac{1}{\theta}\right) \left(\int_{\Omega \cup \partial \Omega}|\nabla u_n|^2 d \mu+\int_{\Omega} h u_n^2 d \mu\right) -C\\
   		&=\left(\frac{1}{2}-\frac{1}{\theta}\right)\left\|u_n\right\|_H^2-C .
   	\end{aligned}
   \end{equation}
   Since $ \theta>2 $, thus $ \{u_n\} $ is bounded. By Lemma 2.1, we can get that  $ \Phi \left(u_n\right) $ satisfies the $ (\mathrm{PS})_c $ condition. The proof is completed.

\end{proof}

We are ready to give the proof of Theorem 1.1.

\begin{proof}[Proof of Theorem 1.1]

  By Lemma $3.1-3.3$, $ \Phi \left(0\right)=0 $ and Lemma 2.2, there exist a function $ u \in H $ such that
 $$
 \Phi \left(u\right)=\inf _{\gamma \in \Gamma} \max _{t \in[0,1]} J(\gamma(t))>0$$
  and $\Phi^{\prime}(u)=0
 $, where
$$
\Gamma=\left\{\gamma \in C([0,1], H): \gamma(0)=0, \gamma(1)=e\right\}.
$$
Since 
$$
   \Phi \left(u\right)=c\geq \delta >0,
$$
then there exist a nontrivial solution $u \in H$ to the equation
(1.7). The proof is completed.
\end{proof}

\section{Proof of the Theorem 1.3 }

In the proof of our theorem, the following lemma is very important.
The following method is similar to the paper \cite{5}.
\begin{lem}
  Presume that the conditions encompassed in Theorem 1.3 are upheld. Suppose that
  $$
  \left\langle \Phi^{\prime}\left(u_n\right), u_n\right\rangle \rightarrow 0 \ (n \rightarrow \infty),\quad \forall\left\{u_n\right\} \in W_0^{1,2}(\Omega),
  $$
  then there exist a subsequence, still denoted by $\left\{u_n\right\}$, such that
  $$
  \Phi\left(t u_n\right) \leq \frac{1+t^2}{2 n}+\Phi\left(u_n\right), \quad \forall \ t>0 .
  $$
\end{lem}
\begin{proof}
		By $\left\langle \Phi^{\prime}\left(u_n\right), u_n\right\rangle \rightarrow 0$ as $n \rightarrow \infty$, we may assume that
	\begin{equation}
	   -\frac{1}{n}<\left\langle \Phi^{\prime}\left(u_n\right), u_n\right\rangle= \int_{\Omega \cup \partial \Omega} \Gamma\left(u_n, u_n\right) d \mu+\int_{\Omega} h u_n^2 d \mu-\int_{\Omega} f\left(x, u_n\right) u_n d \mu <\frac{1}{n}.
	\end{equation}
	The demonstration of this lemma parallels that in \cite{7}.
 The proof is completed.
\end{proof}

\begin{lem} 
  Presume that the conditions encompassed in Theorem 1.3 are upheld. If there is a sequence $\left\{u_n\right\} \in W_0^{1,2}(\Omega                                                                                                                       )$ such that
 \begin{equation}
   \Phi\left(u_n\right) \rightarrow c_0,\ \left(1+\left\|u_n\right\|\right)\left\|\Phi^{\prime}\left(u_n\right)\right\| \rightarrow 0,
 \end{equation}
   then $\left\{u_n\right\}$ has a convergent subsequence.
\end{lem}

\begin{proof} 
  Since (1.13) holds, we only need to prove that $\left\{u_n\right\}$ is bounded in 
  $W_0^{1,2}(\Omega) $.                                                                                                                    If it is not true, we assume $\left\|u_n\right\| \rightarrow+\infty$ as $n \rightarrow \infty$, and set
  $$
  v_n=\frac{\left(4 c_0 \right)^{1 / 2} u_n}{\left\|u_n\right\|},\ v_n^{+}=\frac{\left(4 c_0\right)^{1 / 2} u_n^{+}}{\left\|u_n\right\|}.
  $$
  By Lemma 2.7, then  there is $v \in W_0^{1,2}(\Omega)$ such that 
  \begin{equation}
    \begin{cases}
    	v_n \rightarrow v, v_n^{+} \rightarrow v^{+}, & \text { in } W_0^{1,2}(\Omega), \\ v_n \rightarrow v, v_n^{+} \rightarrow v^{+}, & \text { in } L^2(\Omega), \\ v_n(x) \rightarrow v(x), & 
    	\forall \ x \in \Omega .
    \end{cases}
  \end{equation}

  (1) By $\left(F_6\right)$, there exists a positive constant $ C $ and sufficiently large $ n $,
  \begin{equation}
     \frac{f\left(x,u_n^+\right)}{u_n^+}\geq C, \ \text{uniformly in}  \ x \in \Omega_1=\left\{x \in \Omega: v^+>0\right\} 
  \end{equation}
 
  From (1.12), we derive
  $$
  \int_{\Omega \cup \partial \Omega} \Gamma\left(u_n, u_n\right) d \mu+\int_{\Omega} h u_n^2 d \mu-\int_{\Omega} f\left(x, u_n\right) u_n d \mu=o(1).
  $$
  Coupling this with  (4.4) results in 
  $$
  \begin{aligned}
  	4 c_0 & =\lim_{n \rightarrow \infty} \left\|v_n\right\|^2=\lim _{n \rightarrow \infty} \int_{\Omega} \frac{f\left(x, u_n^{+}\right)}{u_n^{+}}\left(v_n^{+}\right)^2 d \mu \\
  	& \geq C \lim_{n \rightarrow \infty} \int_{\Omega_1}\left(v_n^{+}\right)^2 d \mu=
  	C \int_{\Omega_1}\left(v^{+}\right)^2 d \mu,
  \end{aligned}
  $$
  therefore
  \begin{equation}
    v^{+}=0, \ x \in \Omega.
  \end{equation}
  Thus by (1.13) and (4.5), we have
  $$
  \int_{\Omega} F\left(x, v_n\right) d \mu=\int_{\Omega} F\left(x, v_n^{+}\right) d \mu \rightarrow 0 \ (n \rightarrow \infty),
  $$
  which implies
  \begin{equation}
    \Phi\left(v_n\right) \rightarrow 2 c_0.
  \end{equation}
  
  (2) Let $t_n=\frac{\left(4 c_0 / b\right)^{1 / 2}}{\left\|u_n\right\|}$, then $t_n \rightarrow 0$ ($n \rightarrow \infty$). By Lemma 4.1, we get
   \begin{equation}
     \Phi\left(v_n\right)=\Phi\left(t_n u_n\right) \leq \frac{1+t_n^2}{2 n}+\Phi\left(u_n\right) \rightarrow c_0.
  \end{equation}
  Since $c_0>0$, a contradiction to (4.6). The proof is completed.
 \end{proof}

  \begin{proof}[Proof of Theorem 1.3]
  
 By Lemma 4.1-4.2, $ \Phi \left(0\right)=0 $ and Lemma 2.2, there exist a function $ u \in H $ such that
$$
\Phi \left(u\right)=\inf _{\gamma \in \Gamma} \max _{t \in[0,1]} J(\gamma(t))>0$$
and $\Phi^{\prime}(u)=0
$, where 
$$
  \Gamma=\left\{\gamma \in C([0,1], H): \gamma(0)=0, \gamma(1)=e\right\}.
$$
Noticing that
$$
 \Phi \left(u\right)=c\geq \delta >0,
$$
hence there exist a nontrivial solution $u $. The proof is completed.
 \end{proof}

\section{Proof of Theorem 1.5}

In this section, we are ready to give the proof of Theorem 1.5.

\begin{proof}[Proof of Theorem 1.5]

Take $ \zeta=\frac{1}{2}$. The idea of the proof is predicated on the application of Lemma 2.8 to the functional 
$$
\Phi\left(u\right)=\frac{1}{2}\left\|u\right\|_H^2-\int_{\Omega} F\left(x, u\right) d \mu,
$$
where 
$$
\Theta\left(u\right):=\left\|u\right\|_H^2,
$$
as well as 
$$
\Psi(u):=\int_{\Omega} F(x, u(x)) d \mu
$$
for any $ u \in W_0^{1,2}(\Omega) $.

It is readily apparent that  $\Theta$ is sequentially weakly lower semicontinuous, that is,
$$
\liminf_{n \rightarrow \infty}\left\|u_n\right\|_H \geq\|u\|_H.
$$
Moreover, according to Lemma 2.7, it is evident that  $\Theta$ is coercive in $W_0^{1,2}(\Omega)$.

 There exist  a constant $ \sigma \in[0,1]$, by (1.13), Lemma 2.7 and Hölder inequality, we have
 $$
 \begin{aligned}
 	\left|\Psi\left(u_n\right)-\Psi(u)\right|&=\left|\int_{\Omega} F\left(x, u_n(x)\right) d u-\int_{\Omega} F(x, u(x)) d \mu\right| \\
 	& \leq \int_{\Omega}\left|F\left(x, u_n(x)\right)-F(x, u(x))\right| d \mu \\
 	& =\left|\int_{\Omega} f\left(x, u+\sigma\left(u_n-u\right)\right)\left(u_n-u\right) d \mu\right| \\
 	& \leq \left(\int_{\Omega}\left|f\left(x, u+\sigma\left(u_n-u\right)\right)\right|^{\frac{p-1}{p}}d \mu\right)^{\frac{p-1}{p}}\left(\int_{\Omega}\left|u_n-u\right|^p d \mu\right)^{\frac{1}{p}}
 \end{aligned}
 $$
 
$$
\begin{aligned}
	& \leq C\left(\int_{\Omega}\left(1+\left|u_n\right|^{p-1}+|u|^{p-1}\right)^{\frac{p}{p-1}} d \mu\right)^{\frac{p-1}{p}}\left(\int_{\Omega}\left|u_n-u\right|^p d \mu\right)^{\frac{1}{p}} \\
	& \leq C\left(\int_{\Omega}\left|u_n-u\right|^p\right)^{\frac{1}{p}} \rightarrow 0, \quad n \rightarrow \infty,
	&
\end{aligned}
$$
thus $\Psi$ is sequentially weakly continuous in $W_0^{1,2}(\Omega)$.

According to Lemma 3.2, it is observed that the functional $ \Phi $ is unbounded from below in $W_0^{1,2}(\Omega)$. As per Lemma 3.3, it is understood that the functional $ \Phi $ satisfies the 
$(PS)_c$ condition for any $c \in \mathbb{R}$.

Finally, let $\rho>0$ and
$$
 \chi(\rho):=\inf _{u \in B_\rho} \frac{\left(\sup _{v \in B_\rho} \Psi(v)\right)-\Psi(u)}{\rho-\|u\|_H^2},
$$
where
$$
\mathbb{B}_\rho=\left\{v \in W_0^{1,2}(\Omega):\|v\|_H<\sqrt{\rho}\right\}.
$$
For any $u \in \mathbb{B}_\rho$ and $0 \in \mathbb{B}_\rho$, we have

\begin{equation}
	\begin{aligned}
		\chi(\rho) 
		&	\leq \frac{\sup _{v \in \mathbb{B}_{\rho}}      \Psi(v)}{\rho-\|u\|_H^2}\\
		& \leq \frac{1}{\rho} \sup_{v \in \mathbb{B}_{\rho}} \Psi(v) \\
		& \leq \frac{1}{\rho} \sup _{v \in \overline{\mathbb{B}}_{\rho}}\left|\int_{\Omega} F(x, v(x)) d \mu\right| \\
		& \leq \frac{1}{\rho} \sup_{v \in \overline{\mathbb{B}}_{\rho}} \int_{\Omega}|F(x, v(x))| d \mu .
	\end{aligned}
\end{equation}

Now, assume that the function $ h $ complies with the assumption $ \left(H_1\right) $. Then, if $ v \in \overline{\mathbb{B}}_{\rho} $,
by Lemma 2.7 we get 
\begin{equation}
   |v(x)| \leq \left(\mu_{\min}h_0\right)^{\frac{1}{2}}\|v\| \leq \left(\mu_{\min}h_0\right)^{\frac{1}{2}}\|v\|_H \leq
   \left(\mu_{\min}h_0\right)^{\frac{1}{2}} \sqrt{\rho}\quad \text { for any } x \in \Omega.
\end{equation}
Since
\begin{equation}
	|F(x, v(x))| \leq \max_{\substack{x \in \Omega \\|u| \leq \left(\mu_{\min}h_0\right)^{\frac{1}{2}} \sqrt{\rho}}}\left|\int_0^u f(x, t) d t\right| .
\end{equation}
Hence
\begin{equation}
	\int_{\Omega}|F(x, v(x))| d \mu \leq \max_{\substack{x \in \Omega \\|u| \leq \left(\mu_{\min}h_0\right)^{\frac{1}{2}} \sqrt{\rho}}}\left|\int_0^u f(x, t) d t\right||\Omega|
\end{equation}
for any $ v \in \overline{\mathbb{B}}_{\rho} $.

By $(5.1)$ and $(5.4)$ we have that
$$
\chi(\rho) \leq \frac{1}{\rho} \max_{\substack{x \in \Omega \\|u| \leqslant\left(\mu_{\min} h_0 \rho\right)^{\frac{1}{2}}}}\left|\int_0^u f(x, t) d t\right| \cdot|\Omega| \leq \frac{1}{2(\beta+1)}|\Omega|,
$$
provide $\beta$ satisfies condition (1.16).

If the function $h(x)$ satisfies $\left(H_3\right)$, we can argue in the same way, merely substituting $(5.2)$ with the following inequality
$$
\begin{aligned}
	|v(x)| & \leq\left(\frac{1}{\mu_{\min } h_0}\right)^{\frac{1}{2}}\|v\| \\
	& \leq\left(\mu_{\min } h_0\right)^{\frac{1}{2}}\|v\| \\
	& \leq \frac{\left(\mu_{\min } h_0\right)^{\frac{1}{2}}}{\sqrt{1-\mu_{\min } h_0 \int_{\Omega} h(x) d \mu}}\|v\|_H \\
	& \leq \frac{\left(\mu_{\min } h_0\right)^{\frac{1}{2}}}{\sqrt{1-\mu_{\min } h_0 \int_{\Omega}h(x) d \mu}} \sqrt{\rho} .
\end{aligned}
$$

Considering Lemma 2.8 and incorporating equations $ \left(F_7\right) $ and Lemma 3.3, we deduce that equation (1.7) is guaranteed to have at least two nontrivial solutions one of which lies in $\mathbb{B}_\rho .$ The proof is completed.	
\end{proof}

\section{Proof of Theorem 1.7}

In this section, we prove Theorem 1.7 by Lemma 2.8. Theorem 1.7 is a special form of Theorem 1.5. We prove a conclusion as following.

\begin{proof}[Proof of Theorem 1.7]
Let $ M_0=\kappa \sqrt{\rho}$. If $(H_1)$ holds, we have
$$
2(\beta+1) \max _{(x, u) \in \Omega \times\left[-M_0, M_0\right]}\left|\int_0^u f(x, t) d t\right| \leq \rho=\frac{M_0^2}{\kappa^2},
$$
that is,
$$
\max _{(x, u) \in \Omega \times\left[-M_0, M_0\right]}\left|\int_0^u f(x, t) d t\right| \leq \frac{M_0^2}{2(\beta+1) \kappa^2}=\frac{M_0^2}{2(\beta+1) \mu_{\min} h_0}.
$$

If $(H_3)$ holds, then
$$
2(\beta+1)\max _{(x, u) \in \Omega \times\left[-M_0, M_0\right]}\left|\int_0^u f(x, t) d t\right|\leq \frac{M_0^2\left(1-\mu_{\min} h_0 \int_{\Omega} h(x) d \mu\right)}{\mu_{\min} h_0},
$$
that is,
$$
  \max_{(x, u) \in \Omega \times\left[-M_0, M_0\right]}\left|\int_0^u f(x, t) d t\right| \leq \frac{M_0^2\left(1-\mu_{\min } h_0 \int_{\Omega} h(x) d \mu\right)}{2(\beta+1) \mu_{\min } h_0} 
 \leq \frac{M_0^2}{2(\beta+1) \mu_{\min} h_0}.
$$

Overall, we deduce that
$$
\max_{(x, u) \in \Omega \times\left[-M_0, M_0\right]}\left|\int_0^u f(x, t) d t\right|\leq \frac{M_0^2}{2(\beta+1) \mu_{\min } h_0}.
$$
Thus
$$
1<\beta+1 \leq \frac{M_{0}^{2}}{2{{\mu }_{\min }}{{h}_{0}}\underset{(x,u)\in \Omega \times \left[ -{{M}_{0}},{{M}_{0}} \right]}{\mathop{\max }}\,\left| \int_{0}^{u}{f\left( x,t \right)dt} \right|}.
$$
Hence, applying Theorem 1.5 we obtain that equation (1.7) is guaranteed to have at least two nontrivial solutions one of which lies in 
$$
\mathbb{B}_{\frac{M_0^2}{\mu_{\min} h_0}}:=\left\{u \in W_0^{1,2}(\Omega):\|u\|_H<\sqrt{\frac{M_0^2}{\mu_{\min } h_0}}\right\}.
$$
 The proof is completed.
\end{proof}

\section{Appendix}

This appendix primarily serves as a supplemental elaboration for   $(1.2)-(1.4)$ in the Introduction. Firstly, it elucidates the relationship between (1.2) and (1.3). Secondly, no proof for the general form of  (1.4) is provided in reference \cite{1}, which we will complete in this section.

For any $u: V \rightarrow \mathbb{R}$, the function $\Delta u$ is defined by
$$
\Delta u(x)=\frac{1}{\mu(x)} \sum_y u(y) \mu_{x y}-u(x)=\frac{1}{\mu(x)} \sum_y \mu_{x y}(u(y)-u(x)).
$$
Indeed, by using  $(1.1)$, we have
$$
\begin{aligned}
	\Delta u(x) & =\frac{1}{\mu(x)} \sum_y u(y) \mu_{x y}-u(x) \\
	& =\frac{1}{\mu(x)} \sum_y u(y) \mu_{x y}-\frac{1}{\mu(x)} \sum_y \mu_{x y} u(x) \\
	& =\frac{1}{\mu(x)} \sum_y(u(y)-u(x)) \mu_{x y} .
\end{aligned}
$$
Moreover,

$$
\begin{aligned}
	\Gamma(u, v)(x)= & \frac{1}{2}\{\Delta(u(x) v(x))-u(x) \Delta v(x)-v(x) \Delta u(x)\} \\
	= & \frac{1}{2}\left\{\frac{1}{\mu(x)} \sum_{y \sim x}  \mu_{x y}(u(y) v(y)-u(x) v(x))-u(x) \frac{1}{\mu(x)} \sum_{y \sim x}  \mu_{x y}(v(y)-v(x))\right\} \\
	& -\frac{1}{2} v(x) \frac{1}{\mu(x)} \sum_{y \sim x}  \mu_{x y}(u(y)-u(x)) 
\end{aligned}
$$

$$
\begin{aligned}
	= & \frac{1}{2 \mu(x)}\left\{\sum_{y \sim x}  \mu_{x y}(u(y) v(y)-u(x) v(x)-u(x)(v(y)-v(x))-v(x)(u(y)-u(x)))\right\} \\
	= & \frac{1}{2 \mu(x)}\left\{\sum_{y \sim x}  \mu_{x y}(u(y) v(y)-u(x) v(y)-v(x) u(y)+v(x) u(x))\right\} \\
	= & \frac{1}{2 \mu(x)} \sum_{y \sim x}  \mu_{x y}(u(y)-u(x))(v(y)-v(x)) ,
\end{aligned}
$$
that is, $ (1.4) $ holds.	$\hfill\square$\\

\noindent {\bf \textcolor[rgb]{0.00,0.00,1.00}{Acknowledgements}}

The authors would like to express their sincere
thanks to  the reviewers for their invaluable suggestions and remarks. The paper was supported by NSFC (12271373, 12171326).


\begin{thebibliography}{99}

\bibitem{1} 
 Y. Lin, S. T. Yau, Ricci curvature and eigenvalue estimation on locally finite graphs, Math. Res. Lett., 17 (2010) 345–358.

\bibitem{2} 
 A. Grigoryan, Y. Lin, Y. Yang, Existence of positive solutions to some nonlinear equations on locally finite graphs, Sci. China Math., 60 (2017) 1311–1324. 

\bibitem{3} 
Z. -Q. Wang, On a superlinear elliptic equation, Ann. Inst. Henri Poincaré, Anal. Non Linéaire, 8 (1991) 43-57.

\bibitem{4} 
P. H. Rabinowitz, J. Su, Z. -Q. Wang, Multiple solutions of superlinear elliptions, Rend. Lincei Mat. Appl., 18 (2007) 97-108.

\bibitem{5} 
M. Sun, Z. Yang, H. Cai, Nonexistence and existence of positive solutions for the Kirchhoff type equation, Appl. Math. Lett., 96 (2019)202-207. 

\bibitem{6}  
D. Shou, On a class of nonlinear schrodinger equations on finite graphs, Bull. Aust. Math. Soc., 101 (2020) 477-487.


\bibitem{7} 
G. Li, H. Zhou, Asymptotically linear Dirichlet problem for the p-Laplacian, Nonlin. Anal., 43 (2001) 1043–1055.

\bibitem{8} 
K. -C. Chang, Methods in Nonlinear Analysis, Springer-Verlag, Berlin, 2005.

\bibitem{9} 
K. -C. Chang, Infinite Dimensional Morse Theory and Multiple Solution Problems, Birkh\"auser, Boston, 1993.

\bibitem{10} 
A. Grigoryan, Y. Lin, Y. Yang, Kazdan–Warner equation on graphs, Calc. Var., 55 (2016), 92–113.

\bibitem{11} 
A. Grigoryan, Y. Lin, Y. Yang, Yamabe type equations on graphs, J. Differ. Equ., 261 (2016) 4924–4943.

\bibitem{12} 
A. Ambrosetti, P. H. Rabinowitz, Dual variational methods in critical point theory and applications, J. Funct. Anal., 14 (1973) 349-381.

\bibitem{13} 
K. J. Falconer, J. Hu : Non-linear elliptical equations on the Sierpiński gasket, J. Math. Anal. Appl. 240 (1999) 552–573.

\bibitem{14}
 P. H. Rabinowitz, On a class of nonlinear Schrödinger equations, Z. Angew. Math. Phys. 43 (1992) 270–291.

\bibitem{15}
T. Bartsch, M. Willem, Infinitely many nonradial solutions of a Euclidean scalar field equation, J.
Funct. Anal. 117 (1993) 447–460.

\bibitem{16}
Y. Li, Remarks on a semilinear elliptic equation on $\mathbb{R}^N$, J. Differ. Equ. 74 (1988) 34-49. 

\bibitem{17}
S. Liu, Y. Yang, Multiple solutions of kazdan-warner equation on graphs in the negatice case, Calc. Var. Partial Differential Equations 59(2020), no.5, Paper No. 164, 15 pp.

\bibitem{18} 
N. Zhang, L. Zhao, Convergence of ground state solutions for nonlinear Schrödinger equations on graphs, Sci. China Math. 61 (2018) 1481-1494.

\bibitem{19} 
X. Han, M. Shao, L. Zhao,
Existence and convergence of solutions for nonlinear biharmonic equations on graphs, J. Differ. Equ., 268 (2020) 3936-3961.

\bibitem{20} 
J. Hu, Multiple solutions for a class of nonlinear elliptic equations on the Sierpinski gasket, Sci. China Ser. A 47 (2004) 772–786.

\bibitem{21} 
C. Hua, H. Zhenya, Semilinear elliptic equations on fractal sets, Acta Math. Sci. Ser. B Engl. Ed. 29 (2009) 232–242.

\bibitem{22}
 R. S. Strichartz, Solvability for differential equations on fractals, J. Anal. Math. 96 (2005) 247–267 . 

\bibitem{23}
 B. E. Breckner, D. Repovš, Cs. Varga, On the existence of three solutions for the Dirichlet problem on the Sierpinski gasket, Nonlinear Anal. 73 (2010) 2980–2990.

\bibitem{24}
 G. M. Bisci, D. Repovš, R. Servadei, Nonlinear problems on the Sierpiński gasket, J. Math. Anal. Appl. 452 (2017) 883–895. 

\bibitem{25} 
B. Ricceri, On a classical existence theorem for nonlinear elliptic equations, in: M. Théra (Ed.), Esperimental, Constructive and Nonlinar Analysis, in: CMS Conf. Proc., Canad. Math. Soc., 27 (2000) 275–278.

\bibitem{26}
P. Pucci, J. Serrin, A mountain pass theorem, J. Differential Equations, 60 (1985) 142-149.

\bibitem{27}
B. Ricceri, A general variational principle and some of its applications, J. Comput. Appl. Math. 113 (2000) 401-410.



\end{thebibliography}
  \end{document}